\documentclass[11pt,reqno]{amsart}

\usepackage[T1]{fontenc}
\usepackage{lmodern}
\usepackage{amsmath,amssymb,amsthm,mathtools,mathrsfs}
\usepackage{microtype}
\usepackage[margin=1.12in]{geometry}
\usepackage{aliascnt}
\usepackage[colorlinks=true,linkcolor=blue,citecolor=blue,urlcolor=blue]{hyperref}
\usepackage[nameinlink,capitalize,noabbrev]{cleveref}
\hypersetup{
  pdftitle={Sharp Metric Cotype Inequalities for L1 via Nonlinear Cut Smoothing},
  pdfauthor={Qingjin Cheng, Yue Wang, and Bo Xiang},
  pdfsubject={Metric cotype and nonlinear geometry of L1}
}

\allowdisplaybreaks
\numberwithin{equation}{section}

\newtheorem{theorem}{Theorem}[section]
\newaliascnt{corollary}{theorem}
\newtheorem{corollary}[corollary]{Corollary}
\aliascntresetthe{corollary}
\newaliascnt{lemma}{theorem}
\newtheorem{lemma}[lemma]{Lemma}
\aliascntresetthe{lemma}
\theoremstyle{remark}
\newtheorem{remark}[theorem]{Remark}

\newcommand{\R}{\mathbb R}
\newcommand{\C}{\mathbb C}
\newcommand{\Z}{\mathbb Z}
\newcommand{\T}{\mathbb T}
\newcommand{\E}{\mathbb E}
\newcommand{\1}{\mathbf 1}

\newcommand{\abs}[1]{\left|#1\right|}
\newcommand{\norm}[1]{\left\|#1\right\|}
\newcommand{\ip}[2]{\left\langle #1,#2\right\rangle}

\title[Sharp metric cotype for \(L_1\)]{Sharp Metric Cotype Inequalities for \(L_1\)
via Nonlinear Cut Smoothing}

\author{Qingjin Cheng}
\author{Yue Wang}
\author{Bo Xiang}
\address{School of Mathematical Sciences, Xiamen University,
Xiamen 361005, China}
\email[Q. Cheng]{qjcheng@xmu.edu.cn}
\email[Y. Wang]{19020260158074@stu.xmu.edu.cn}
\email[B. Xiang]{Xiangbo7501@gmail.com}
\thanks{The work was supported by NSFC (No.~12071389).}
\date{}


\begin{document}

\begin{abstract}
We determine the optimal order of the \(p\)-moment torus inequality for
\(L_1\): it is \(m^p n^{(1-p/2)_+}+n\), with comparison constants
independent of \(p,m,n\) after taking \(p\)-th roots. Consequently, for
every \(q\geq2\) and \(1\leq p\leq q\), the corresponding metric cotype
inequality holds at the sharp scale \(m=O(n^{1/q})\). In particular,
the quadratic inequality holds with \(m=O(\sqrt n)\), answering a
question of Mendel and Naor. The proof combines finite cut
representations with a nonlinear smoothing estimate and an exact
ternary rounding identity.
\end{abstract}

\maketitle

\section{Introduction and main results}

Recall that a Banach space \(X\) has Rademacher cotype
\(q\in[2,\infty)\) if there is a constant \(C<\infty\) such that, for
every finite sequence \(x_1,\ldots,x_n\in X\),
\begin{equation}\label{eq:Rademacher-cotype}
 \left(\sum_{j=1}^n\norm{x_j}_X^q\right)^{1/q}
 \leq
 C\left(\E_\varepsilon
 \norm{\sum_{j=1}^n\varepsilon_jx_j}_X^q\right)^{1/q},
\end{equation}
where \(\varepsilon\) is uniform on \(\{-1,1\}^n\). Motivated by the
Ribe program, Mendel and Naor introduced metric cotype as a metric
analogue of \eqref{eq:Rademacher-cotype} \cite{MendelNaor2008}.

For an even integer \(m\), write \(\Z_m=\Z/m\Z\), and let
\(e_1,\ldots,e_n\) denote the standard coordinate vectors; we use the same
notation in \(\R^n\) and in \(\Z_m^n\). A metric
space \((\mathcal M,d_{\mathcal M})\) has metric cotype \(q>0\) with
constant \(\Gamma<\infty\) if, for every \(n\in\mathbb N\), there is
an even integer \(m=m(n)\) such that every
\(f:\Z_m^n\to\mathcal M\) satisfies
\begin{equation}\label{eq:metric-cotype-definition}
\sum_{j=1}^n\E_x
d_{\mathcal M}\!\left(f\!\left(x+\frac m2e_j\right),f(x)\right)^q
\leq
\Gamma^qm^q\E_{x,\varepsilon}
d_{\mathcal M}(f(x+\varepsilon),f(x))^q,
\end{equation}
where \(x\) is uniform on \(\Z_m^n\) and \(\varepsilon\) is uniform on
\(\{-1,0,1\}^n\), with \(x\) and \(\varepsilon\) independent. Mendel and
Naor proved that a Banach space has metric cotype \(q\) if and only if it
has Rademacher cotype \(q\) \cite{MendelNaor2008}. For fixed \(n\) and
\(\Gamma\), let \(m_q(\mathcal M;n,\Gamma)\) be the smallest even \(m\)
for which \eqref{eq:metric-cotype-definition} holds for every
\(f:\Z_m^n\to\mathcal M\), with the convention that it is infinite if
no such \(m\) exists.

In \eqref{eq:metric-cotype-definition}, \(q\) is both the cotype index
and the moment exponent. Mendel and Naor separated these roles
\cite[Definition~1.3]{MendelNaor2008}. For \(1\leq p\leq q\), let
\(m_q^{(p)}(\mathcal M;n,\Gamma)\) be the smallest even integer \(m\)
for which every \(f:\Z_m^n\to\mathcal M\) satisfies
\begin{equation}\label{eq:weak-metric-cotype-definition}
\sum_{j=1}^n\E_x
d_{\mathcal M}\!\left(f\!\left(x+\frac m2e_j\right),f(x)\right)^p
\leq
\Gamma^pm^pn^{1-p/q}\E_{x,\varepsilon}
d_{\mathcal M}(f(x+\varepsilon),f(x))^p.
\end{equation}
The same \(+\infty\) convention applies.
Here \(p\) is the moment exponent and \(q\) the cotype index; the
factor \(n^{1-p/q}\) is the normalization appearing in the linear theory. When
\(p=q\), this is exactly \eqref{eq:metric-cotype-definition}, so
\(m_q^{(q)}(\mathcal M;n,\Gamma)=m_q(\mathcal M;n,\Gamma)\). For
\(p<q\), \eqref{eq:weak-metric-cotype-definition} is the weak metric
cotype \(q\) inequality with exponent \(p\).

If \(\mathcal M\) contains at least two points, the random two-point
test of \cite[Lemma~2.3]{MendelNaor2008} gives, for
\(1\leq p\leq q\),
\begin{equation}\label{eq:universal-scale-lower-bound}
 m_q^{(p)}(\mathcal M;n,\Gamma)
 \geq \frac{n^{1/q}}{\Gamma}.
\end{equation}
Thus \(n^{1/q}\) is the smallest possible order of the scaling
parameter. This scale is also the one needed in the embedding
applications of \cite{MendelNaor2008,GiladiMendelNaor2011}.

Mendel and Naor proved that if \(X\) is \(K\)-convex and has
Rademacher cotype \(q\), then, for every \(1\leq p\leq q\), there are
constants \(\Gamma,K<\infty\), depending only on \(X,p,q\), such that
\(m_q^{(p)}(X;n,\Gamma)\leq K n^{1/q}\) for every \(n\in\mathbb N\);
this order is optimal \cite[Section~4]{MendelNaor2008}. Without
\(K\)-convexity, their general construction yielded the weaker order
\(n^{2+1/q}\), and they asked whether every Banach space of Rademacher
cotype \(q\) has the sharp order \(n^{1/q}\); see
\cite[Section~8, item~1]{MendelNaor2008}. This question is specific to
Banach spaces, since qualitative metric cotype does not imply sharp
scaling for arbitrary metric spaces \cite{EskenazisMendelNaor2019}.
Giladi, Mendel, and Naor subsequently improved the \(p=q\) bound to the
order \(n^{1+1/q}\), and noted that their argument extends to the
corresponding moment variants \cite[p.~167]{GiladiMendelNaor2011}.
For cotype \(2\), this yields a scale of order \(n^{3/2}\). They
singled out \(L_1\) and the Schatten class \(S_1\) as two of the
simplest Banach spaces for which a sharp quadratic metric cotype
inequality remained unknown \cite[Section~1.3]{GiladiMendelNaor2011}.

For \(L_1\), Mendel and Naor asked whether there exist constants
\(\Gamma,K<\infty\) such that \(m_2(L_1;n,\Gamma)\leq K\sqrt n\) for
every \(n\in\mathbb N\); see
\cite[Section~8, item~2, p.~294]{MendelNaor2008}.

We answer this question and, more generally, determine the optimal order
of the underlying torus inequality, with comparison constants that are
universal after taking \(p\)-th roots.

For \(1\leq p<\infty\), even \(m\geq4\), and \(n\in\mathbb N\), define
\begin{equation}\label{eq:def-Cpmn}
 \mathsf C_p(m,n;L_1)
 =
 \sup_f
 \frac{\displaystyle
 \sum_{j=1}^n \E_x
 \norm{f\left(x+\frac m2e_j\right)-f(x)}_1^p}
 {\displaystyle
 \E_{x,\varepsilon}\norm{f(x+\varepsilon)-f(x)}_1^p},
\end{equation}
where the supremum is taken over all real or complex \(L_1(\Omega,\mu)\)
spaces and all mappings \(f:\Z_m^n\to L_1(\Omega,\mu)\) with nonzero
denominator. Equivalently, it suffices to take the fixed range space
\(L_1[0,1]\), since every finite \(L_1\)-metric embeds isometrically into
it.

\begin{theorem}\label{thm:optimal}
For every \(p\geq1\), every even integer \(m\geq4\), and every
\(n\in\mathbb N\), every mapping \(f:\Z_m^n\to L_1\) satisfies
\begin{align}
 &\left(\sum_{j=1}^n\E_x
 \norm{f(x+\tfrac m2e_j)-f(x)}_1^p\right)^{1/p}
 \notag\\
 &\qquad\leq
 \left(\frac{9\sqrt{210}}{16}\,m n^{(1/p-1/2)_+}
       +\frac94 n^{1/p}\right)
 \left(\E_{x,\varepsilon}\norm{f(x+\varepsilon)-f(x)}_1^p\right)^{1/p}.
 \label{eq:upper-root}
\end{align}
Moreover,
\begin{equation}\label{eq:optimal-two-regime}
 \frac{m^pn^{(1-p/2)_+}+n}{(2\pi)^p}
 \leq
 \mathsf C_p(m,n;L_1)
 \leq
 11^p\left(m^pn^{(1-p/2)_+}+n\right).
\end{equation}
Here \((a)_+=\max\{a,0\}\).
\end{theorem}

For \(1\leq p\leq2\), the first term in
\eqref{eq:optimal-two-regime} is \(m^pn^{1-p/2}\); for \(p\geq2\), it
is \(m^p\). The term \(n\) is necessary for every \(p\geq1\).

\begin{corollary}\label{cor:sharp-metric-cotype}
For every \(2\leq q<\infty\), every \(1\leq p\leq q\), every even
\(m\geq4\), every \(n\leq m^q\), and every mapping
\(f:\Z_m^n\to L_1\),
\begin{align}
 &\sum_{j=1}^n \E_x
 \norm{f\left(x+\frac m2e_j\right)-f(x)}_1^p
 \notag\\
 &\qquad\leq
 11^p m^p n^{1-p/q}
 \E_{x,\varepsilon}\norm{f(x+\varepsilon)-f(x)}_1^p.
 \label{eq:sharp-metric-cotype-L1}
\end{align}
Consequently,
\[
 m_q^{(p)}(L_1;n,11)\leq 6 n^{1/q}
 \qquad(q\geq2,\ 1\leq p\leq q,\ n\in\mathbb N).
\]
In particular,
\[
 m_2(L_1;n,11)\leq 6\sqrt n
 \qquad(n\in\mathbb N).
\]
\end{corollary}

Together with \eqref{eq:universal-scale-lower-bound}, this proves that
the order \(n^{1/q}\) in \cref{cor:sharp-metric-cotype} is optimal.

We briefly describe the argument. The finite \(L_1\)-metric induced by
\(f\) is first represented as a finite weighted sum of cut metrics, whose
indicators are extended to piecewise-constant functions on the continuous
torus. Each indicator is then smoothed by a product kernel and composed
with the cubic \(P(s)=3s^2-2s^3\). An exact ternary rounding identity
matches the resulting kernel average with the uniform increment on
\(\{-1,0,1\}^n\). A score-function identity, followed by Bessel's
inequality, controls the gradient in terms of the Bernoulli variance,
while \(P'(s)=6s(1-s)\) converts the variance term into the approximation
error. Summing over the cuts and integrating along coordinate paths of length \(1/2\)
yields the upper bound. The lower bounds come from Hilbertian and
antiperiodic examples.

The key point is that this smoothing is nonlinear and adapted to the
chosen cut representation, rather than a fixed linear convolution on
arbitrary Banach-valued maps. Hence the construction lies outside the
framework of \cite{GiladiMendelNaor2011}; see \cref{sec:smoothing}. The
argument is specific to cut representations of finite \(L_1\)-metrics and
does not directly extend to general finite-cotype Banach spaces or to
\(S_1\).

The nonlinear cut-smoothing construction is developed in
\cref{sec:cut-smoothing}; the upper and lower bounds and the metric cotype
consequence are proved in \cref{sec:main-proofs}.

\section{Nonlinear cut smoothing}\label{sec:cut-smoothing}

We first replace the finite \(L_1\)-metric by finitely many weighted cut
coordinates and extend them cellwise to the continuous torus. We then
construct the exact ternary kernel and prove the nonlinear smoothing
estimate used in the upper bound.

\subsection{Finite cut representations}\label{sec:cuts}

We may work with real \(L_1\) spaces. Indeed, for every \(z\in\C\),
\begin{equation}\label{eq:complex-realification}
 |z|=\frac14\int_0^{2\pi}
 \abs{\operatorname{Re}(e^{-i\theta}z)}\,d\theta.
\end{equation}
Integrating this identity gives a real-linear isometric embedding of every
complex \(L_1\) space into a real one. For the finite families considered
below, one may also restrict to a \(\sigma\)-finite measurable support.

For a set \(S\) and \(A\subseteq S\), let \(\1_A:S\to\{0,1\}\)
denote the indicator of \(A\), and define the associated cut metric on
\(S\) by
\[
 \delta_A(x,y)=\abs{\1_A(x)-\1_A(y)}.
\]
We use the standard cut representation of finite \(L_1\)-metrics; see,
for example, \cite{DezaLaurent1997}. We include the short argument for
completeness.

\begin{lemma}\label{lem:finite-cuts}
Let \(S\) be a finite set and let \(f:S\to L_1(\Omega,\mu)\). Then there are
subsets \(A_1,\ldots,A_N\subseteq S\) and weights
\(\lambda_1,\ldots,\lambda_N\geq0\), with \(\sum_r\lambda_r<\infty\), such that
\begin{equation}\label{eq:finite-cut-representation}
 \norm{f(x)-f(y)}_1
 =
 \sum_{r=1}^N\lambda_r\delta_{A_r}(x,y)
 \qquad(x,y\in S).
\end{equation}
\end{lemma}

\begin{proof}
Fix \(x_0\in S\) and replace \(f\) by \(f-f(x_0)\). We may also
restrict the underlying measure space to the \(\sigma\)-finite support
of \(\sum_{x\in S}|f(x)|\). For
\((\omega,t)\in\Omega\times\R\), let
\[
 A_{\omega,t}=\{x\in S:f(x)(\omega)>t\}.
\]
By the layer-cake representation and Tonelli's theorem,
\[
 \norm{f(x)-f(y)}_1
 =
 \int_{\Omega\times\R}
 \delta_{A_{\omega,t}}(x,y)\,d\mu(\omega)\,dt.
\]

For fixed \(\omega\), the cut \(A_{\omega,t}\) is nontrivial only when
\(t\) lies between
\(\min_{x\in S}f(x)(\omega)\) and
\(\max_{x\in S}f(x)(\omega)\). The length of this interval is
\[
 \max_{x\in S}f(x)(\omega)-\min_{x\in S}f(x)(\omega),
\]
which, since \(f(x_0)=0\), is bounded by
\(\sum_{x\in S}|f(x)(\omega)|\). Hence the set of
\((\omega,t)\) producing a nontrivial cut has finite
\((\mu\otimes dt)\)-measure.

Since \(S\) is finite, only finitely many nonempty proper subsets
\(A\subsetneq S\) can occur. For each such \(A\), define
\[
 \lambda_A
 =
 (\mu\otimes dt)\bigl\{(\omega,t):A_{\omega,t}=A\bigr\}.
\]
The preceding estimate gives \(\sum_A\lambda_A<\infty\). Since
\(A_{\omega,t}\) takes only finitely many nontrivial values, the
layer-cake representation becomes
\[
 \norm{f(x)-f(y)}_1
 =
 \sum_A\lambda_A\delta_A(x,y).
\]
Omitting the zero-weight terms and relabeling the remaining cuts as
\(A_1,\ldots,A_N\) proves the lemma.
\end{proof}

We apply \cref{lem:finite-cuts} with \(S=\Z_m^n\). Let
\(A_1,\ldots,A_N\subseteq\Z_m^n\) and
\(\lambda_1,\ldots,\lambda_N\geq0\) be the subsets and weights given
by the lemma. The constant case is immediate, so we assume that \(f\)
is nonconstant. We may then omit the zero-weight terms and assume that
\(\lambda_r>0\) for all \(r\). Write
\(\lambda=(\lambda_1,\ldots,\lambda_N)\) and equip \(\R^N\) with the
weighted \(\ell_1\)-norm
\[
 \norm{a}_{\ell_1^N(\lambda)}
 =
 \sum_{r=1}^N\lambda_r|a_r|.
\]
Then the map \(x\mapsto(\1_{A_r}(x))_{r=1}^N\) preserves all
pairwise distances.

Let \(\T=\R/\Z\), represented by \([0,1)\), and identify
\(\Z_m\) with \(\{0,\ldots,m-1\}\). For \(x\in\Z_m^n\), set
\[
 Q_x
 =
 \prod_{j=1}^n
 \left[\frac{x_j}{m},\frac{x_j+1}{m}\right).
\]
For \(r=1,\ldots,N\), define
\[
 \widetilde A_r=\bigcup_{x\in A_r}Q_x,
\]
and let
\begin{equation}\label{eq:F-cell-extension}
 F(z)=\bigl(\1_{\widetilde A_r}(z)\bigr)_{r=1}^N
 \in\ell_1^N(\lambda).
\end{equation}
Thus, if \(z\in Q_x\) and \(w\in Q_y\), then
\begin{equation}\label{eq:cell-distance}
 \norm{F(z)-F(w)}_{\ell_1^N(\lambda)}
 =
 \norm{f(x)-f(y)}_1.
\end{equation}

\subsection{An exact ternary smoothing kernel}\label{sec:kernel}

We use the following compactly supported function on \(\R\):
\begin{equation}\label{eq:psi-definition}
 \psi(s)=84|s|^5(1-|s|)^2\1_{[-1,1]}(s).
\end{equation}

\begin{lemma}\label{lem:kernel-identities}
The function \(\psi\) is an even probability density on \(\R\). It is
continuously differentiable and compactly supported, so
\(\psi\in C_c^1(\R)\). Moreover,
\begin{equation}\label{eq:psi-moment}
 \int_{\R}|s|\psi(s)\,ds=\frac23.
\end{equation}
Its Fisher information is
\begin{equation}\label{eq:psi-Fisher}
 \int_{\{\psi>0\}}\frac{\psi'(s)^2}{\psi(s)}\,ds=70.
\end{equation}
\end{lemma}

\begin{proof}
By symmetry,
\begin{align*}
 \int_{\R}\psi(s)\,ds
 &=168\int_0^1s^5(1-s)^2\,ds=1,\\
 \int_{\R}|s|\psi(s)\,ds
 &=168\int_0^1s^6(1-s)^2\,ds=\frac23.
\end{align*}
The factors \(|s|^5\) and \((1-|s|)^2\) also show that \(\psi\) and
its first derivative are continuous at \(0\) and at \(\pm1\). Thus
\(\psi\in C_c^1(\R)\). On \(0<s<1\), we have
\[
 \frac{\psi'(s)}{\psi(s)}=\frac5s-\frac2{1-s}.
\]
Using symmetry once more, we obtain
\[
 \int_{\{\psi>0\}}\frac{\psi'(s)^2}{\psi(s)}\,ds
 =168\int_0^1s^3(5-7s)^2\,ds=70.
\]
\end{proof}

The value \(2/3\) in \eqref{eq:psi-moment} will make the three integer
increments \(-1,0,1\) equally likely. We obtain these increments by
adding a displacement with density \(\psi\) to a uniformly chosen point
of \([0,1)\) and taking the integer part.

\begin{lemma}\label{lem:ternary-rounding}
Let \(U\) be uniformly distributed on \([0,1)\), and let \(V\) have
density \(\psi\), independently of \(U\). Then
\(J=\lfloor U+V\rfloor\) is uniformly distributed on \(\{-1,0,1\}\).
\end{lemma}

\begin{proof}
Since \(U\in[0,1)\) and \(|V|\leq1\) almost surely, \(J\) can only
take the values \(-1,0,1\).

For a fixed displacement \(v\in[0,1]\), we have
\(\lfloor U+v\rfloor=1\) precisely when \(U\in[1-v,1)\). This
interval has length \(v\), so
\[
 \mathbb P(J=1\mid V=v)=v.
\]
Similarly, when \(v\in[-1,0]\), we have
\(\lfloor U+v\rfloor=-1\) precisely when \(U\in[0,-v)\). Hence
\[
 \mathbb P(J=-1\mid V=v)=-v.
\]
We now average over \(V\). Since its distribution is symmetric and
\(\E|V|=2/3\),
\[
 \mathbb P(J=1)=\mathbb P(J=-1)=\frac12\E|V|=\frac13.
\]
The remaining probability is also \(1/3\).
\end{proof}

For independent copies \((U_j,V_j)\) of \((U,V)\), the variables
\(J_j=\lfloor U_j+V_j\rfloor\) are independent and uniformly
distributed on \(\{-1,0,1\}\). In our application, the starting point
in \(\Z_m^n\) is chosen independently of these pairs. The vector
\(J=(J_1,\ldots,J_n)\) is therefore independent of that point.

For even \(m\geq4\), we rescale \(\psi\) to \(\T\). Representing
\(t\in\T\) by its element in \([-1/2,1/2)\), we set
\begin{equation}\label{eq:scaled-kernel}
 \psi_m(t)=m\psi(mt).
\end{equation}
Then \(\psi_m\in C^1(\T)\) and is a probability density on \(\T\).
A change of variables gives
\begin{equation}\label{eq:scaled-Fisher}
 \int_{\{\psi_m>0\}}\frac{\psi_m'(t)^2}{\psi_m(t)}\,dt
 =70m^2.
\end{equation}

\subsection{Nonlinear Fisher smoothing}\label{sec:smoothing}

We now come to the main analytic estimate. Our aim is to replace the
binary map from \cref{sec:cuts} by a smooth map, while controlling both
the approximation error and the derivatives by the same local averaging
error.

Recall that \(\T=\R/\Z\). Fix an even integer \(m\geq4\) and
\(n,N\in\mathbb N\). We use the probability density \(\psi_m\) from
\eqref{eq:scaled-kernel}; its Fisher information is given by
\eqref{eq:scaled-Fisher}.

Let \(\lambda_1,\ldots,\lambda_N>0\), write
\(\lambda=(\lambda_1,\ldots,\lambda_N)\), and for
\(a=(a_1,\ldots,a_N)\in\R^N\) set
\[
 \norm{a}_{\lambda}=\sum_{r=1}^N\lambda_r|a_r|.
\]
We write \(\ell_1^N(\lambda)\) for \(\R^N\) equipped with this norm.
Let
\[
 F=(F_1,\ldots,F_N):\T^n\longrightarrow\{0,1\}^N
\]
be measurable, where \(F_r\) denotes the \(r\)-th coordinate of \(F\).
In the application below, \(F\) and the weights \(\lambda_r\) are the
ones obtained from the cut representation in \cref{sec:cuts}.

Let \(Z=(Z_1,\ldots,Z_n)\) be a \(\T^n\)-valued random vector whose
coordinates are independent and have density \(\psi_m\). For
\(z\in\T^n\) and \(r=1,\ldots,N\), define
\[
 u_r(z)=\E_Z F_r(z+Z).
\]
Thus \(0\leq u_r(z)\leq1\). We now compose these averages with the cubic
\begin{equation}\label{eq:cubic-P}
 P(s)=3s^2-2s^3,\qquad 0\leq s\leq1.
\end{equation}
The map \(P\) fixes the endpoints \(0\) and \(1\), and
\(P'(0)=P'(1)=0\). Define
\[
 G:\T^n\longrightarrow\R^N,\qquad
 G(z)=\bigl(P(u_1(z)),\ldots,P(u_N(z))\bigr).
\]
Finally, let
\begin{equation}\label{eq:Rz-definition}
 R(z)=\E_Z\norm{F(z+Z)-F(z)}_{\lambda},\qquad z\in\T^n.
\end{equation}
The next lemma gives the two estimates that we will use in the proof of
the upper bound.

\begin{lemma}\label{prop:smoothing}
Fix an even integer \(m\geq4\) and \(n,N\in\mathbb N\). Let
\(\lambda=(\lambda_1,\ldots,\lambda_N)\in(0,\infty)^N\), let
\(F:\T^n\to\{0,1\}^N\) be measurable, and let \(G\) and \(R\) be
defined as above. Then \(G\in C^1(\T^n;\ell_1^N(\lambda))\), and for
every \(z\in\T^n\),
\begin{equation}\label{eq:smoothing-approximation}
 \norm{F(z)-G(z)}_{\lambda}\leq\frac98R(z),
\end{equation}
and
\begin{equation}\label{eq:smoothing-gradient}
 \left(\sum_{j=1}^n\norm{\partial_jG(z)}_{\lambda}^2\right)^{1/2}
 \leq \frac{9\sqrt{210}}8\,mR(z).
\end{equation}
\end{lemma}

\begin{proof}
Let
\[
 k_m(w)=\prod_{j=1}^n\psi_m(w_j),
 \qquad w=(w_1,\ldots,w_n)\in\T^n,
\]
be the density of \(Z\). For each \(r=1,\ldots,N\), the function \(u_r\)
is the convolution of the binary coordinate \(F_r\) with \(k_m\):
\[
 u_r(z)=\int_{\T^n}F_r(w)k_m(w-z)\,dw.
\]
Since \(F_r\) is bounded and \(k_m\in C^1(\T^n)\), we may
differentiate under the integral. Hence \(u_r\in C^1(\T^n)\) for
\(r=1,\ldots,N\). Since \(P\) is a polynomial, it follows that
\[
 G\in C^1(\T^n;\ell_1^N(\lambda)).
\]

We first estimate how far \(G\) is from the original binary map \(F\).
For \(b\in\{0,1\}\) and \(0\leq s\leq1\),
\begin{equation}\label{eq:cubic-endpoint}
 |b-P(s)|\leq\frac98|b-s|.
\end{equation}
Indeed, when \(b=0\),
\[
 \frac98-(3s-2s^2)=2\left(s-\frac34\right)^2\geq0,
\]
and the case \(b=1\) follows from \(1-P(s)=P(1-s)\). Since
\(F_r(z)\in\{0,1\}\),
\[
 |F_r(z)-u_r(z)|
 =\E_Z|F_r(z)-F_r(z+Z)|.
\]
Therefore
\begin{align*}
 \norm{F(z)-G(z)}_{\lambda}
 &\leq \frac98\sum_{r=1}^N\lambda_r|F_r(z)-u_r(z)|\\
 &=\frac98R(z),
\end{align*}
which proves \eqref{eq:smoothing-approximation}.

We now use the Fisher information of \(\psi_m\) to control the
derivatives of \(u_r\). Define the score function
\[
 S_m(t)=-\frac{\psi_m'(t)}{\psi_m(t)}
 \qquad\text{on }\{\psi_m>0\},
\]
and set \(S_m(t)=0\) elsewhere. Since
\(S_m\psi_m=-\psi_m'\) almost everywhere, differentiating the
convolution and changing variables gives
\begin{equation}\label{eq:score-identity}
 \partial_j u_r(z)
 =\E_Z\bigl[F_r(z+Z)S_m(Z_j)\bigr].
\end{equation}

The score has mean zero, while its second moment is exactly the
rescaled Fisher information:
\[
 \E_Z S_m(Z_j)=0,
 \qquad
 \E_Z S_m(Z_j)^2=70m^2.
\]
We may therefore subtract the mean \(u_r(z)\) from \(F_r(z+Z)\) in
\eqref{eq:score-identity} and write
\[
 \partial_j u_r(z)
 =\E_Z\bigl[(F_r(z+Z)-u_r(z))S_m(Z_j)\bigr].
\]

Because the coordinates of \(Z\) are independent, the normalized
scores
\[
 \frac{S_m(Z_1)}{m\sqrt{70}},\ldots,
 \frac{S_m(Z_n)}{m\sqrt{70}}
\]
form an orthonormal family in \(L_2\). On the other hand,
\(F_r(z+Z)\) is a Bernoulli random variable with mean \(u_r(z)\), so
\[
 \E_Z|F_r(z+Z)-u_r(z)|^2
 =u_r(z)(1-u_r(z)).
\]
Bessel's inequality now gives
\begin{equation}\label{eq:Bessel-u}
 \left(\sum_{j=1}^n|\partial_j u_r(z)|^2\right)^{1/2}
 \leq
 m\sqrt{70}\sqrt{u_r(z)(1-u_r(z))}.
\end{equation}

The estimate \eqref{eq:Bessel-u} leaves us with the square root of a
Bernoulli variance. This is precisely the quantity that the cubic
\(P\) is designed to absorb. Indeed, since \(P'(s)=6s(1-s)\),
\begin{equation}\label{eq:cubic-key-bound}
 P'(s)\sqrt{s(1-s)}
 \leq
 \frac{9\sqrt3}{8}\min\{s,1-s\},
 \qquad 0\leq s\leq1.
\end{equation}
For \(0<s\leq1/2\), this amounts to
\[
 6s^{1/2}(1-s)^{3/2}\leq\frac{9\sqrt3}{8},
\]
whose left-hand side attains its maximum at \(s=1/4\).
The case \(1/2\leq s<1\) follows by symmetry.

Combining \eqref{eq:Bessel-u}, \eqref{eq:cubic-key-bound}, and the chain
rule, we obtain
\[
 \left(\sum_{j=1}^n|\partial_jP(u_r(z))|^2\right)^{1/2}
 \leq \frac{9\sqrt{210}}8\,m\min\{u_r(z),1-u_r(z)\}.
\]
We now sum over \(r\). Minkowski's inequality in \(\ell_2^n\) gives
\begin{align*}
 \left(\sum_{j=1}^n\norm{\partial_jG(z)}_{\lambda}^2\right)^{1/2}
 &\leq\sum_{r=1}^N\lambda_r
 \left(\sum_{j=1}^n|\partial_jP(u_r(z))|^2\right)^{1/2}\\
 &\leq \frac{9\sqrt{210}}8\,m\sum_{r=1}^N\lambda_r
 \min\{u_r(z),1-u_r(z)\}.
\end{align*}
Finally, the binary value \(F_r(z)\in\{0,1\}\) gives
\[
 \min\{u_r(z),1-u_r(z)\}
 \leq |u_r(z)-F_r(z)|
 =\E_Z|F_r(z+Z)-F_r(z)|.
\]
After summing with the weights \(\lambda_r\), the right-hand side is
exactly \(R(z)\). This proves \eqref{eq:smoothing-gradient}.
\end{proof}

\begin{remark}
The smoothing used here differs from the fixed linear convolution
framework of \cite{GiladiMendelNaor2011}. We first represent the finite
\(L_1\)-metric by binary cut coordinates, smooth each coordinate, and
then apply the nonlinear map \(P\). Thus the smoothing depends on the
chosen cut representation and is nonlinear before the coordinates are
reassembled. We note that the lower bound in
\cite[Proposition~4.1]{GiladiMendelNaor2011} already uses an
\(L_1\)-valued test map.
\end{remark}

\section{Proofs of the main results}\label{sec:main-proofs}

We first prove the upper bound in the torus inequality, then establish
the matching lower bounds. The metric cotype consequence is derived at
the end of the section.

\subsection{The upper bound}\label{sec:upper}

\begin{proof}[Proof of the upper bound in \cref{thm:optimal}]
We pass to the continuous torus so that long displacements can be
estimated along coordinate paths. The smoothing error and the derivatives
will both be controlled by the same local average.

Fix \(1\leq p<\infty\), an even integer \(m\geq4\), and \(n\in\mathbb N\).
Let \(f:\Z_m^n\to L_1(\Omega,\mu)\), where \((\Omega,\mu)\) is a measure
space. Using the realification in \cref{sec:cuts}, we may replace a complex
\(L_1\) space by a real one without changing any distances.
The constant case is immediate, so we assume that \(f\) is
nonconstant. Taking only the cuts of positive weight in
\cref{lem:finite-cuts}, we may choose \(N\in\mathbb N\),
cuts \(A_1,\ldots,A_N\subseteq\Z_m^n\), and weights
\(\lambda=(\lambda_1,\ldots,\lambda_N)\in(0,\infty)^N\).
For \(a\in\R^N\), write
\(\norm{a}_{\lambda}=\sum_{r=1}^N\lambda_r|a_r|\).
We use normalized Haar measure on \(\T^n\) and the cells
\[
 Q_x=\prod_{j=1}^n\left[\frac{x_j}{m},\frac{x_j+1}{m}\right)
 \subseteq\T^n\qquad(x\in\Z_m^n).
\]
Define \(F:\T^n\to\{0,1\}^N\subseteq\ell_1^N(\lambda)\) by
\(F(z)=(\1_{A_r}(x))_{r=1}^N\) for \(z\in Q_x\).
The cut representation gives
\begin{equation}\label{eq:cell-distance-upper}
 \norm{F(z)-F(w)}_{\lambda}=\norm{f(x)-f(y)}_1
 \qquad(z\in Q_x,\ w\in Q_y).
\end{equation}

We now use the smoothed map \(G\in C^1(\T^n;\ell_1^N(\lambda))\)
and the \(\T^n\)-valued random increment \(Z\) from \cref{sec:smoothing}.
With \(R:\T^n\to[0,\infty)\) given by
\[
 R(z)=\E_Z\norm{F(z+Z)-F(z)}_{\lambda},
\]
\cref{prop:smoothing} gives, for every \(z\in\T^n\),
\begin{align*}
 \norm{F(z)-G(z)}_{\lambda}&\leq\frac98R(z),\\
 \left(\sum_{j=1}^n\norm{\partial_jG(z)}_{\lambda}^2\right)^{1/2}
 &\leq\frac{9\sqrt{210}}8\,mR(z).
\end{align*}

We estimate a long displacement of \(F\) by following \(G\) along a
coordinate path. For \(z\in\T^n\) and \(j=1,\ldots,n\), the fundamental
theorem of calculus and the approximation bound give
\[
 \begin{split}
 \norm{F(z+\tfrac12e_j)-F(z)}_{\lambda}
 &\leq\int_0^{1/2}
 \norm{\partial_jG(z+te_j)}_{\lambda}\,dt\\
 &\quad+\frac98\bigl(R(z)+R(z+\tfrac12e_j)\bigr).
 \end{split}
\]
The dimension factor enters when we pass from a square sum of derivatives
to a \(p\)-sum. The comparison of \(\ell_p^n\) and \(\ell_2^n\) norms gives
\[
 \begin{split}
 \left(\sum_{j=1}^n\norm{\partial_jG(z)}_{\lambda}^p\right)^{1/p}
 &\leq n^{(1/p-1/2)_+}
 \left(\sum_{j=1}^n\norm{\partial_jG(z)}_{\lambda}^2\right)^{1/2}\\
 &\leq\frac{9\sqrt{210}}8\,
 m n^{(1/p-1/2)_+}R(z).
 \end{split}
\]
Minkowski's integral inequality and translation invariance on \(\T^n\)
now give
\begin{align}
 &\left(\sum_{j=1}^n\int_{\T^n}
 \norm{F(z+\tfrac12e_j)-F(z)}_{\lambda}^p\,dz\right)^{1/p}
 \notag\\
 &\quad\leq\frac12
 \left(\sum_{j=1}^n\int_{\T^n}
 \norm{\partial_jG(z)}_{\lambda}^p\,dz\right)^{1/p}
 +\frac94n^{1/p}\norm{R}_{L_p(\T^n)}
 \notag\\
 &\quad\leq
 \left(\frac{9\sqrt{210}}{16}\,
 m n^{(1/p-1/2)_+}+\frac94n^{1/p}\right)
 \norm{R}_{L_p(\T^n)}.
 \label{eq:continuous-upper}
\end{align}

We now recover the discrete local average. Choose \(z\in\T^n\) uniformly
and independently of \(Z\), and let \(x\in\Z_m^n\) be its cell index.
Conditional on \(x\), the point \(z\) is uniform in \(Q_x\).
By \cref{lem:ternary-rounding}, applied coordinatewise at scale \(1/m\),
we have \(z+Z\in Q_{x+\varepsilon}\), where
\(\varepsilon\in\{-1,0,1\}^n\) is uniform and independent of \(x\).
Since \(x\) is also uniform, \eqref{eq:cell-distance-upper} gives
\begin{equation}\label{eq:local-exact-identification}
 \int_{\T^n}\E_Z\norm{F(z+Z)-F(z)}_{\lambda}^p\,dz
 =\E_{x,\varepsilon}\norm{f(x+\varepsilon)-f(x)}_1^p.
\end{equation}
Jensen's inequality gives
\begin{equation}\label{eq:R-Jensen}
 \norm{R}_{L_p(\T^n)}^p
 \leq\E_{x,\varepsilon}\norm{f(x+\varepsilon)-f(x)}_1^p.
\end{equation}

For the long increments, translation by \(e_j/2\) sends \(Q_x\) onto
\(Q_{x+(m/2)e_j}\), since \(m\) is even. Each cell has measure \(m^{-n}\), so
\begin{equation}\label{eq:long-exact-identification}
 \int_{\T^n}\norm{F(z+\tfrac12e_j)-F(z)}_{\lambda}^p\,dz
 =\E_x\norm{f(x+\tfrac m2e_j)-f(x)}_1^p
 \qquad(j=1,\ldots,n).
\end{equation}
With these identifications, \eqref{eq:continuous-upper} becomes
\begin{align}
 &\left(\sum_{j=1}^n\E_x
 \norm{f(x+\tfrac m2e_j)-f(x)}_1^p\right)^{1/p}
 \notag\\
 &\qquad\leq
 \left(\frac{9\sqrt{210}}{16}\,
 m n^{(1/p-1/2)_+}+\frac94n^{1/p}\right)
 \left(\E_{x,\varepsilon}
 \norm{f(x+\varepsilon)-f(x)}_1^p\right)^{1/p}.
\end{align}
Both \(m n^{(1/p-1/2)_+}\) and \(n^{1/p}\) are at most
\(\bigl(m^pn^{(1-p/2)_+}+n\bigr)^{1/p}\), while
\(9\sqrt{210}/16+9/4<11\). Raising the preceding inequality to the
power \(p\) yields
\[
 \sum_{j=1}^n\E_x\norm{f(x+\tfrac m2e_j)-f(x)}_1^p
 \leq11^p\bigl(m^pn^{(1-p/2)_+}+n\bigr)
 \E_{x,\varepsilon}\norm{f(x+\varepsilon)-f(x)}_1^p.
\]
This proves the upper bound.
\end{proof}

\subsection{Lower-bound constructions}\label{sec:lower}

We use Euclidean cycle maps for the term involving \(m\). For the term
\(n\), we use a binary-valued map that changes under every half-period
coordinate shift.

Fix \(1\leq p<\infty\), an even integer \(m\geq4\), and
\(n\in\mathbb N\). Throughout this subsection, \(x\) and
\(\varepsilon\) are independent and uniform on \(\Z_m^n\) and
\(\{-1,0,1\}^n\), respectively. We obtain lower bounds for
\(\mathsf C_p(m,n;L_1)\) by evaluating the ratio in \eqref{eq:def-Cpmn}
on these maps.

\begin{lemma}\label{prop:Hilbert-lower}
\[
 \mathsf C_p(m,n;L_1)
 \geq \pi^{-p}m^pn^{(1-p/2)_+}.
\]
\end{lemma}

\begin{proof}
We begin with Euclidean-valued maps and realize their distances in
\(L_1\). For \(d\in\mathbb N\), let \(\gamma_d\) be the standard
Gaussian probability measure on \(\R^d\). The linear map
\[
 \ell_2^d\longrightarrow L_1(\R^d,\gamma_d),
 \qquad
 v\longmapsto
 \bigl[\omega\longmapsto\sqrt{\pi/2}\,\ip{v}{\omega}\bigr]
\]
is an isometry, since
\[
 \sqrt{\frac\pi2}\int_{\R^d}\abs{\ip{v}{\omega}}\,d\gamma_d(\omega)
 =\norm{v}_2
 \qquad(v\in\ell_2^d).
\]
We may thus compute the ratios using Euclidean norms.

To separate half-period shifts from unit shifts, define
\(u_m:\Z_m\to\ell_2^2\) by
\begin{equation}\label{eq:cycle-map}
 u_m(k)=\frac1{2\sin(\pi/m)}
 \left(\cos\frac{2\pi k}{m},\sin\frac{2\pi k}{m}\right)
 \qquad(k\in\Z_m).
\end{equation}
This normalization gives
\begin{equation}\label{eq:cycle-distances}
 \norm{u_m(k\pm1)-u_m(k)}_2=1,
 \qquad
 \norm{u_m(k+\tfrac m2)-u_m(k)}_2
 =\frac1{\sin(\pi/m)}\geq\frac m\pi.
\end{equation}

For \(1\leq p\leq2\), define \(U:\Z_m^n\to\ell_2^{2n}\) by
\[
 U(x)=\bigl(u_m(x_1),\ldots,u_m(x_n)\bigr).
\]
The coordinate pairs are orthogonal, so
\begin{align*}
 \sum_{j=1}^n\E_x\norm{U(x+\tfrac m2e_j)-U(x)}_2^p
 &=\frac{n}{\sin^p(\pi/m)}\geq n\left(\frac m\pi\right)^p,\\
 \norm{U(x+\varepsilon)-U(x)}_2^2
 &=\sum_{j=1}^n\varepsilon_j^2.
\end{align*}
Since \(p/2\leq1\), Jensen's inequality gives
\[
 \E_{x,\varepsilon}\norm{U(x+\varepsilon)-U(x)}_2^p
 \leq\left(\sum_{j=1}^n\E_\varepsilon\varepsilon_j^2\right)^{p/2}
 =\left(\frac{2n}{3}\right)^{p/2}\leq n^{p/2}.
\]
The ratio is at least \(\pi^{-p}m^pn^{1-p/2}\), as required.

For \(2\leq p<\infty\), use the one-coordinate map
\(x\mapsto u_m(x_1)\) from \(\Z_m^n\) to \(\ell_2^2\).
The numerator is \(\sin^{-p}(\pi/m)\). The denominator is \(2/3\),
since only \(\varepsilon_1=\pm1\) contributes, with distance \(1\).
Its isometric image in \(L_1\) gives
\[
 \mathsf C_p(m,n;L_1)
 \geq\frac{3}{2\sin^p(\pi/m)}\geq\pi^{-p}m^p.\qedhere
\]
\end{proof}

\begin{lemma}\label{prop:n-lower}
\[
 \mathsf C_p(m,n;L_1)\geq2n.
\]
\end{lemma}

\begin{proof}
We use a sign that is constant on each half-cycle. With representatives
\(k\in\{0,\ldots,m-1\}\), define \(s_m:\Z_m\to\{-1,1\}\) by
\[
 s_m(k)=
 \begin{cases}
  1,&0\leq k<m/2,\\
  -1,&m/2\leq k<m.
 \end{cases}
\]
A half-period shift reverses this sign:
\begin{equation}\label{eq:antiperiodic-sign}
 s_m(k+\tfrac m2)=-s_m(k)
 \qquad(k\in\Z_m).
\end{equation}
We now define \(f:\Z_m^n\to\{0,1\}\subseteq\R\) by
\[
 f(x)=\frac12\left(1-\prod_{j=1}^ns_m(x_j)\right).
\]
We regard \(\R\) as \(L_1\) of a one-point probability space.
The identity \(f(x+\tfrac m2e_j)=1-f(x)\) shows that every long
increment has length \(1\), so
\[
 \sum_{j=1}^n\E_x\abs{f(x+\tfrac m2e_j)-f(x)}^p=n.
\]

To compute the local average, let \(k\) and \(\eta\) be independent
and uniform on \(\Z_m\) and \(\{-1,0,1\}\), respectively, and set
\[
 \rho_m=\E_{k,\eta}s_m(k+\eta)s_m(k).
\]
Exactly two cyclic edges change the sign, so
\[
 \E_k s_m(k+1)s_m(k)=1-\frac4m.
\]
The shift \(-1\) has the same average by translation invariance.
Including the zero increment gives
\begin{equation}\label{eq:rho-bound}
 \rho_m=\frac13\left(1+2\left(1-\frac4m\right)\right)
 =1-\frac8{3m}\in\left[\frac13,1\right).
\end{equation}

Since \(f\) takes only the values \(0\) and \(1\),
\[
 \abs{f(x+\varepsilon)-f(x)}^p
 =\frac12\left(1-\prod_{j=1}^n
 s_m(x_j+\varepsilon_j)s_m(x_j)\right).
\]
The pairs \((x_j,\varepsilon_j)\) are independent across coordinates,
so
\[
 \E_{x,\varepsilon}\abs{f(x+\varepsilon)-f(x)}^p
 =\frac{1-\rho_m^n}{2}\in\left(0,\frac12\right).
\]
The defining ratio is at least \(2n\).
\end{proof}

\begin{proof}[Proof of the lower bound in \cref{thm:optimal}]
The two examples give
\begin{align*}
 \mathsf C_p(m,n;L_1)
 &\geq\max\left\{\pi^{-p}m^pn^{(1-p/2)_+},\,2n\right\}\\
 &\geq\frac12\left(\pi^{-p}m^pn^{(1-p/2)_+}+2n\right)\\
 &\geq\frac{m^pn^{(1-p/2)_+}+n}{(2\pi)^p}.
\end{align*}
The last inequality uses \(\pi^{-p}/2\geq(2\pi)^{-p}\), which holds
for \(p\geq1\).
\end{proof}

\subsection{Sharp metric cotype for \texorpdfstring{\(L_1\)}{L1}}
\label{sec:consequences}

\begin{proof}[Proof of \cref{cor:sharp-metric-cotype}]
We choose the torus scale so that both terms in the upper bound have
the metric cotype normalization. Fix \(2\leq q<\infty\),
\(1\leq p\leq q\), \(n\in\mathbb N\), and an even integer \(m\geq4\)
with \(n\leq m^q\). Let \((\Omega,\mu)\) be an arbitrary measure
space and \(f:\Z_m^n\to L_1(\Omega,\mu)\) a mapping. We take \(x\) and
\(\varepsilon\) to be independent and uniform on \(\Z_m^n\) and
\(\{-1,0,1\}^n\), respectively, and write
\(e_1,\ldots,e_n\in\Z_m^n\) for the standard coordinate vectors.

The assumptions \(q\geq2\), \(p\leq q\), and \(n^{1/q}\leq m\) give
\[
 n^{(1/p-1/2)_+}\leq n^{1/p-1/q},
 \qquad
 n^{1/p}=n^{1/q}n^{1/p-1/q}\leq m n^{1/p-1/q}.
\]
Using these comparisons in the \(p\)-th root estimate
\eqref{eq:upper-root}, we obtain
\begin{align*}
 &\left(\sum_{j=1}^n\E_x
 \norm{f(x+\tfrac m2e_j)-f(x)}_1^p\right)^{1/p}\\
 &\qquad\leq
 \left(\frac{9\sqrt{210}}{16}\,m n^{(1/p-1/2)_+}
       +\frac94 n^{1/p}\right)
 \left(\E_{x,\varepsilon}\norm{f(x+\varepsilon)-f(x)}_1^p\right)^{1/p}\\
 &\qquad\leq
 11m n^{1/p-1/q}
 \left(\E_{x,\varepsilon}\norm{f(x+\varepsilon)-f(x)}_1^p\right)^{1/p},
\end{align*}
since \(9\sqrt{210}/16+9/4<11\). Taking \(p\)-th powers proves
\eqref{eq:sharp-metric-cotype-L1}.

Recall that \(m_q^{(p)}(L_1;n,11)\) is the smallest positive even
integer \(m\) for which this inequality holds for every \(L_1\)-valued
map on \(\Z_m^n\). We now choose \(m\) to be the least even integer
not smaller than \(\max\{4,n^{1/q}\}\). The preceding estimate applies
at this scale, and \(n\geq1\) gives
\[
 m_q^{(p)}(L_1;n,11)\leq m
 <\max\{4,n^{1/q}\}+2\leq6n^{1/q}.
\]
At \(p=q=2\), this yields
\(m_2(L_1;n,11)=m_2^{(2)}(L_1;n,11)\leq6\sqrt n\).
The universal lower bound \eqref{eq:universal-scale-lower-bound}
reads \(m_q^{(p)}(L_1;n,11)\geq n^{1/q}/11\), so the order
\(n^{1/q}\) is sharp.
\end{proof}

\end{document}